\documentclass{compositio}

\usepackage{amsfonts, amssymb, amsmath, latexsym}
\usepackage[all]{xy}

\newtheorem{thm}{Theorem}[section]
\newtheorem{defn}[thm]{Definition}
\newtheorem{prop}[thm]{Proposition}

\newtheorem{lemma}[thm]{Lemma}
\newtheorem{remark}[thm]{Remark}

\newcommand{\Z}{\mathrm{Z}}
\newcommand{\D}{\mathrm{D}}

\newcommand{\E}{\mathrm{E}}
\newcommand{\LL}{\mathrm{L}}
\newcommand{\X}{\mathrm{X}}
\newcommand{\V}{\mathrm{V}}
\newcommand{\pv}{\mathbf{P}(\V)}

\begin{document}

\title{Automorphisms of Drinfeld half-spaces over a
    finite field}
\author{Bertrand R\'emy}
\email{remy@math.univ-lyon1.fr}
\address{Universit\'e de Lyon\\Universit\'e Lyon1-CNRS\\Institut Camille Jordan - UMR5208\\43 bd. du 11 novembre 1918 \\ F-69622 Villeurbanne cedex}
\author{Amaury Thuillier}
\email{thuillier@math.univ-lyon1.fr}
\address{Universit\'e de Lyon\\Universit\'e Lyon1-CNRS\\Institut Camille Jordan - UMR5208\\43 bd. du 11 novembre 1918 \\ F-69622 Villeurbanne cedex}
\author{Annette Werner}
\email{werner@mathematik.uni-frankfurt.de}
\address{Institut f\"ur Mathematik \\Goethe-Universit\"at Frankfurt\\Robert-Mayer-Str. 6-8\\ D-60325 Frankfurt a.M.}


\begin{abstract}
We show that the automorphism group of Drinfeld's half-space over a finite field is the projective linear group of the underlying vector space. The proof of this result uses analytic geometry in the sense of Berkovich over the  finite field equipped with the trivial valuation. We also take into account extensions of the base field.
\end{abstract}

\maketitle

\section*{Introduction}
\label{s - intro}
In this note we determine the automorphism group of Drinfeld's half-spaces over a finite field.
Given a finite-dimensional vector space ${\rm V}$ over a finite field $k$, the Drinfeld half-space $\Omega(\V)$ is defined as the complement of all $k$-rational hyperplanes in the projective space $\pv$; it is an affine algebraic variety over $k$.  We show that every $k$-automorphism of $\Omega({\rm V})$ is induced by a $k$-automorphism of $\pv$. Hence the automorphism group of $\Omega(\V)$ is equal to ${\rm PGL}({\V})$.

More generally, for an arbitrary field extension ${\rm K}$ of $k$, we prove that the natural injection of ${\rm PGL}({\rm V})$ into ${\rm Aut}_{\rm K}(\Omega({\rm V}) \otimes_k {\rm K})$ is an isomorphism. Our result answers a question of Dat, Orlik and Rapoport \cite[p.~338]{DOR} which was motivated by the analogous statement for Drinfeld half-spaces over a  non-Archimedean local field (with non-trivial absolute value). 

Drinfeld defined his $p$-adic upper half-spaces in \cite{drin}.
They are the founding examples of the theory of period domains \cite{razi}.
Analogs of period domains over finite fields have been studied by Rapoport in \cite{rap}.
They are open subvarieties of flag varieties characterized by a semi-stability condition.
Recently, they have been studied  by Rapoport, Orlik and others, see e.g. \cite{orl01}, \cite{orlra}.
A good introduction is given in the book \cite{DOR}.

Over local non-Archimedean fields with non-trivial absolue value, Drinfeld half-spaces are no longer algebraic varieties and must be defined in the context of analytic geometry. In this setting, it was shown by Berkovich that every automorphism is induced by a projective linear transformation \cite{Berkovich_Drinfeld}. This was generalized to products of Drinfeld half-spaces by Alon \cite{Alon}, who also pointed out and corrected a discrepancy in Berkovich's proof. Berkovich's strategy is based on the fact that in the case of a local non-Archimedean ground field with non-trivial absolute value, the Bruhat-Tits building of the group ${\rm PGL}({\V})$ is contained in $\Omega(\V)$ as the subset of points satisfying a natural maximality condition. This implies that every automorphism of $\Omega(\V)$ induces an automorphism of the Bruhat-Tits building, and with some further work (see \cite{Alon}) one can show the claim. 

One could in fact use a similar strategy in order to determine the
automorphism group of $\Omega(\V)$ over a finite field. Indeed, if we
endow the finite ground field with the trivial absolute value and look
at the corresponding Berkovich analytic space $\Omega(\V)^{\rm an}$,
by \cite{Berkovich_Book}, the \emph{vectorial building} associated to
the group ${\rm PGL}({\V})$ is contained in $\Omega(\V)^{\rm
  an}$. We believe that one can then follow Berkovich's and Alon's arguments to deduce that this automorphism comes from an element of ${\rm PGL}({\rm V})$.

However, in this note, we adopt a slightly different, and maybe more natural, viewpoint. Thereby, we want to highlight that the true content of this theorem is about \emph{extension} of automorphisms, and that it has in fact very little to do with buildings, see Remark \ref{remark-complex}.
Our approach is the following. We consider the space ${\rm X}$ obtained by
blowing up all $k$-rational linear subspaces of the projective space
$\pv$. Irreducible components of the boundary divisor correspond
bijectively to linear subspaces of $\pv$. Moreover, a family of
components has non-empty intersection if and only if the corresponding
linear subspaces form a flag. We use Berkovich analytic geometry to
prove in Proposition \ref{prop-analytic} that every automorphism of
$\Omega(\V)$ preserves the set of discrete valuations on the function
field induced by boundary components of ${\rm X}$. Hence by
Proposition \ref{prop-valuations} it extends to an automorphism of
${\rm X}$. By taking a closer look at the Chow ring of ${\rm X}$ in section 3, we deduce that this automorphism preserves the set of discrete valuations corresponding to hyperplanes, which allows us to conclude that it induces an automorphism of the projective space.

\smallskip

\begin{acknowledgements} We thank Vladimir Berkovich for aquainting us
with the results in \cite{Alon}. We also thank Carlo Gasbarri for suggesting an alternative argument for the last step of our proof, and St\'ephane Lamy for an interesting discussion on birational geometry. We are grateful to the referee for a number of helpful corrections and suggestions how to make the paper more accessible.
\end{acknowledgements}

\section{Automorphisms of Drinfeld's half-spaces}

\label{section-auto}

Let $k$ be a finite field and let ${\rm V}$ be a $k$-vector space.
We denote by $\mathbf{P}({\rm V})$ the projective scheme ${\rm Proj}\left({\rm Sym}^\bullet {\rm V} \right)$ and define the $k$-scheme $\Omega(\V)$ as the complement of all
(rational) hyperplanes in $\mathbf{P}({\rm V})$:
$$\Omega(\V)= \mathbf{P}({\rm V}) \ - \bigcup_{{\tiny \begin{array}{c} {\rm W} \subset {\rm V} \\
{\rm dim}~{\rm W}=1 \end{array}}} \mathbf{P}({\rm V}/{\rm W}).$$
For every field extension ${\rm K}/k$  we denote by $\V_{\rm K} = \V
\otimes_k {\rm K}$ the induced vector space over ${\rm K}$. Then the
base change $\Omega(\V)_{\rm K} = \Omega(\V) \otimes_k {\rm K}$ is the
complement of all $k$-rational hyperplanes in $\mathbf{P}({\rm V}_{\rm
  K}) =  \mathbf{P}(\V) \otimes_k {\rm K}$. 

\smallskip

The main result of this note is the following.

\begin{thm}\label{main-theorem} Let ${\rm V}$ be a vector space of
  finite dimension over a finite field $k$. 
\begin{itemize}
\item[(i)]~The restriction map
$${\rm PGL}({\V}) = \mathrm{Aut}_k\bigl(\mathbf{P}({\rm V})\bigr)
  \rightarrow {\rm Aut}_k \bigl(\Omega(\V) \bigr),\ \
\varphi \mapsto \varphi_{|\Omega(\V)}$$
is an isomorphism.
Equivalently, every $k$-automorphism of \ $\Omega(\V)$ extends to a $k$-automorphism of $\mathbf{P}({\rm V})$.
\item[(ii)]~For every field extension  ${\rm K}/k$  the natural map 
$${\rm PGL}({\rm V}) \longrightarrow {\rm Aut}_{\rm
  K}\bigl(\Omega(\V)_{\rm K}\bigr)$$  
is an isomorphism. Equivalently, every ${\rm K}$-automorphism of  $\Omega(\V)_{\rm K}$ comes by base change from a $k$-automorphism of $\mathbf{P}({\rm V})$.
\end{itemize}
\end{thm}

\vskip1mm This result holds trivially if ${\rm dim} {\rm V} \leqslant 1$, for then $\Omega({\rm V}) = \mathbf{P}({\rm V})$. From now on, we assume that ${\rm V}$ has dimension at least $2$ and we set $n = {\rm dim} {\rm V}-1$.

The proof combines analytic geometry in the sense of Berkovich with  algebraic arguments.
As a first step we show that every $k$-automorphism of $\Omega(\V)$ can be extended to an automorphism of the $k$-scheme ${\rm X}$ we get by blowing-up all linear subspaces of $\mathbf{P}(\V)$.  For this step we use Berkovich analytic geometry over the field $k$ endowed with the trivial absolute value. The second step is of an algebraic nature and consists in checking that this automorphism of ${\rm X}$ is induced by a $k$-automorphism of $\mathbf{P}({\rm V})$. Here we analyze the geometry of the boundary divisor more closely and use an induction argument.

\smallskip

Given a proper subvector space ${\rm W}$ of $\V$, applying ${\rm
  Proj}$ to the natural map ${\rm Sym}^\bullet ({\rm
  V})\twoheadrightarrow {\rm Sym}^\bullet ({\rm V}/{\rm W})$ leads to
a closed immersion $\mathbf{P}({\rm V}/{\rm W}) \hookrightarrow
\mathbf{P}({\rm V})$ whose image ${\rm L}$ is called a \emph{linear subspace}~of $\mathbf{P}({\rm V})$.
Such a subscheme is said to be trivial if ${\rm L} = \varnothing$ or ${\rm L} = \mathbf{P}({\rm V})$; it is called a \emph{hyperplane}~if it is of codimension 1.
We denote by $\mathcal{L}^i({\rm V})$ the set of linear subspaces of dimension $i$ in $\mathbf{P}({\rm V})$, and by $\displaystyle \mathcal{L}({\rm V}) = \bigcup_{0 \leqslant i \leqslant n-1} \mathcal{L}^i({\rm V})$ the set of non-trivial linear subspaces.

\begin{defn}
\label{defi - distinguished divisors}
We denote by $\pi : {\rm X} \rightarrow \mathbf{P}({\rm V})$ the blow-up of $\mathbf{P}({\rm V})$ along the full hyperplane arrangement. To be precise, ${\rm X}$ is defined as
$$\xymatrix{{\rm X} = {\rm X}_{n-1} \ar@{->}[r]^{\pi_{n-1}} & {\rm X}_{n-2} \ar@{->}[r] & \ldots
\ar@{->}[r] & {\rm X}_1 \ar@{->}[r]^{\pi_1} & {\rm X}_0
\ar@{->}[r]^{\pi_0\hskip6mm} & {\rm X}_{-1} = \mathbf{P}({\rm V})}$$
with
$$\pi = \pi_0\circ \pi_1 \circ \ldots \circ \pi_{n-1},$$
where $\pi_i$ denotes the blow-up of \ ${\rm X}_{i-1}$ along the strict transforms of linear subspaces of $\mathbf{P}({\rm V})$ of dimension $i$.
\end{defn}

The scheme ${\rm X}$ is projective and smooth over $k$.
It contains $\Omega(\V)$ as an open dense subscheme since each $\pi_i$ induces an isomorphism over $\Omega(\V)$. We write ${\rm D} = {\rm X} - \Omega(\V)$ for the complement.

Note  that $\pi_{n-1}$ is an isomorphism and that the strict transforms of two distinct linear subspaces ${\rm L}, {\rm L}'\subset~\mathbf{P}({\rm V})$ of dimension $i$ in ${\rm X}_{i-1}$ are disjoint since (the strict transform of) ${\rm L} \cap {\rm L}'$ has been previously blown-up.

Each non-trivial linear subspace ${\rm L} \subset \mathbf{P}({\rm V})$ defines a smooth and irreducible hypersurface ${\rm  E}_{\rm L}$ in ${\rm X}$ as follows. If $\LL$ has dimension $i$, its strict transform by $\pi_0\circ \pi_1 \circ \ldots \circ \pi_{i-1}$ in ${\rm X}_{i-1}$ (by convention $\LL$ itself if it is a point) is blown-up under the map $\pi_i: \X_i \rightarrow \X_{i-1}$ to give rise to a hypersurface ${\rm E}_{\rm L}^{(i)}$ in ${\rm X}_i$. The (codimension 1) subscheme ${\rm E}_{\rm L}$ of $\X$ is then the
strict transform of ${\rm E}_{\rm L}^{(i)}$ by $\pi_{i+1} \circ \ldots \circ \pi_{n-1}$. The induced map ${\rm E}_{\rm L} \rightarrow {\rm E}_{\rm L}^{(i)}$ coincides with the blow-up of ${\rm E}_{\rm L}^{(i)}$ along the hypersurface arrangement induced by hyperplanes of $\mathbf{P}({\rm V})$ containing ${\rm L}$. We have an alternative description of $\E_\LL$ as the closure $$\displaystyle \overline{\pi^{-1}\Bigl(\LL - \bigcup_{{\tiny \begin{array}{c} {\rm L}' \in \mathcal{L}({\rm V}) \\ \LL' \subsetneq \LL \end{array}}} \LL'\Bigr)}$$ taken in $\X$.

It follows from the construction of ${\rm X}$ that the boundary divisor $\D$ is the union of all hypersurfaces $\E_\LL$, i.e. we have  
$${\rm D}
  =  \pi^{-1}\Bigl(\bigcup_{{\tiny \begin{array}{c} {\rm W} \subset
        {\rm V} \\ {\rm dim}~{\rm W}=1 \end{array}}} \mathbf{P}({\rm
    V}/{\rm W})\Bigr)=\bigcup_{\rm L} {\rm E}_{\rm L}.$$
Two  components ${\rm E}_{\rm L}, {\rm E}_{\rm L'}$ have non-empty
intersection if and only if ${\rm L} \subset {\rm L}'$ or ${\rm L}'
\subset {\rm L}$. Indeed, if none of the inclusions holds, then ${\rm
  L}$ and ${\rm L}'$ intersect along a smaller linear subspace, say of
dimension $i$, and the strict transforms of ${\rm L}$ and ${\rm L}'$
in ${\rm X}_i$ are disjoint. It follows that a familly of components
has non-empty intersection if and only if it is indexed by linear
subspaces lying in a flag. We define the stratum ${\rm Z}_{\mathcal{F}}$ corresponding
to a flag $\mathcal{F}$ by:
$${\rm Z}_{\mathcal{F}} = \bigcup_{{\rm L} \in \mathcal{F}} {\rm
  E}_{\rm L} - \bigcup_{{\rm L}' \notin \mathcal{F}} {\rm E}_{{\rm L}'}.$$ 
  
\smallskip
\vskip2mm \begin{lemma} \label{lemma-blow-up} The divisor ${\rm D}$ has simple normal
  crossings. Moreover, if
  \ ${\rm Z} = {\rm Z}_{\mathcal{F}}$ is the stratum corresponding to the flag $\mathcal{F}$,
  then $${\rm U}_{\rm Z} = {\rm X} - \bigcup_{{\rm L} \notin
    \mathcal{F}} {\rm E}_{\rm L}$$ is an affine open subset of \ ${\rm X}$ containing ${\rm Z}$ as a closed subset.
\end{lemma}  

\begin{proof} We start by considering a complete
flag $\mathcal{F} = ({\rm L}_0, \ldots, {\rm L}_n)$. In order to get an  explicit description of ${\rm U}_{\rm Z}$ in
this case, we first compare ${\rm X}$ to the blow-up ${\rm Y}$ of
$\mathbf{P}({\rm V})$ along $\mathcal{F}$. To be precise, we define:
$$\xymatrix{p : {\rm Y} = {\rm Y}_{n-1} \ar@{->}[r]^{p_{n-1}} & {\rm Y}_{n-2} \ar@{->}[r] & \ldots
  \ar@{->}[r] & {\rm Y}_1 \ar@{->}[r]^{p_1} & {\rm Y}_0
  \ar@{->}[r]^{p_0\hskip6mm} & {\rm Y}_{-1} = \mathbf{P}({\rm V})}$$
where $p_i$ denotes the blow-up of ${\rm Y}_{i-1}$ along the strict
transform of ${\rm L}_i$. By the universal property of
blow-up, there exists a (unique) morphism of towers
$f_\bullet : {\rm X}_\bullet \rightarrow {\rm Y}_\bullet$. 

\vskip1mm Now, we want to show that $f$ identifies ${\rm U}_{\rm Z}$
with the complement ${\rm W}_{\rm Z}$ in ${\rm Y}$ of the strict
transforms of all linear subspaces not contained in
$\mathcal{F}$. Note that ${\rm W}_{\rm Z}$ is also the complement of
the strict transform of all hyperplanes distinct from ${\rm L}_{n-1}$. We
argue by induction along the towers of blow-ups. For every $i \in \{-1, \ldots, n-1\}$, we define two open subsets ${\rm U}_i \subset
{\rm X}_i$ and ${\rm W}_i \subset {\rm Y}_i$ as follows: 
\begin{itemize}
\item[-] ${\rm U}_{-1}
= {\rm W}_{-1}$ is the complement in $\mathbf{P}({\rm V})$ of all
$0$-dimensional linear subspaces distinct from ${\rm L}_0$;
\item[-] if $0 \leqslant i \leqslant n-2$, then ${\rm U}_i$
  (resp. ${\rm W}_i$) is the complement in $\pi_i^{-1}({\rm U}_{i-1})$
  (resp. in $p_i^{-1}({\rm W}_{i-1})$) of the strict transforms of all
  $(i+1)$-dimensional linear subspaces ${\rm L} \subset
  \mathbf{P}({\rm V})$ not in $\mathcal{F}$;
\item[-] ${\rm U}_{n-1} = \pi_{n-1}^{-1}({\rm U}_{n-2})$ and ${\rm
  W}_{n-1} = p_{n-1}^{-1}({\rm W}_{n-2})$.
\end{itemize}
Arguing by induction on $i$, we see that ${\rm U}_i = f_i^{-1}({\rm
  W}_i)$, and that $f_i$
induces an isomorphism between ${\rm U}_i$ and ${\rm
  W}_i$ respecting the restrictions of exceptional divisors. It is
clear that $${\rm U}_{n-1} = {\rm U}_{\rm Z} = {\rm X} - \bigcup_{{\rm L} \notin
  \mathcal{F}} {\rm E}_{\rm L}.$$ On the other hand, we claim that ${\rm W}_{n-1}$
coincides with ${\rm W}_{\rm Z}$. The inclusion ${\rm W}_{n-1}
\subset {\rm W}_{\rm Z}$ is obvious. For every point $y \in {\rm Y} -
        {\rm W}_{n-1}$ 
        there exists an index
        $i \in \{-1,\ldots, n-2\}$ such that the image $y_i$ of $y$ in
        ${\rm Y}_i$ lies in the strict transform of a $(i+1)$-dimensional linear subspace ${\rm L}
        \subset \mathbf{P}({\rm V})$ distinct from ${\rm
          L}_{i+1}$. Let us consider a hyperplane ${\rm H}$ which contains
        ${\rm L}$. By construction,
        $y_i$ is contained in the strict transform of ${\rm H}$ in
        ${\rm Y}_i$. Since ${\rm L}_j \not\subset {\rm
          H}$ for $j \in \{i, \ldots, n-2\}$, the subspaces ${\rm
          L}_j$ and ${\rm H}$ are transverse. Blowing-up along some
         smooth subschemes can only decrease the order of contact,
         hence the strict transform $\widetilde{\rm H}$ of
        ${\rm H}$ in ${\rm Y}_j$ is transverse to the center of $p_{j+1}$.
        This implies that the strict transform of ${\rm H}$ in ${\rm Y}_{j+1}$ coincides with the inverse image of $\widetilde{\rm
          H}$ in ${\rm Y}_{j+1}$. It follows that $y$ belongs to the
        strict transform of ${\rm H}$ in ${\rm Y}$, and thus $y
        \in {\rm Y} - {\rm W}_{\rm Z}$. This proves the converse
        inclusion ${\rm W}_{\rm Z} \subset {\rm W}_{n-1}$.
     
\vskip1mm
Given a basis $(e_0, e_1,\ldots, e_n)$ of ${\rm V}$ such that ${\rm L}_i = {\rm Z}(e_{i+1}, \ldots, e_n)$ for
every $i \in \{0,\ldots, n-1\}$, we have a commutative diagram $$\xymatrix{{\rm
    Spec}(k[t_1,\ldots,t_n]) \ar@{^{(}->}[r]^{\hskip 15mm j}
  \ar@{->}[d]_q & {\rm Y} \ar@{->}[d]^p \\ {\rm Spec}(k[x_1,\ldots,
    x_n]) \ar@{^{(}->}[r] & \mathbf{P}({\rm V})}$$ where the horizontal
 arrows are open immersions identifying $t_1,\ldots, t_n$
 (resp. $x_1,\ldots, x_n$) with the rational functions
 $e_1/e_0, \ldots,  e_n/e_{n-1}$ (resp. $e_1/e_0, \ldots,
 e_n/e_0$) and where $q$ is the morphism defined by $q^*(x_i) =
 \prod_{j \leqslant i} t_j$. 

 Via $j$, the open subscheme ${\rm
   W}_{\rm Z}$ of ${\rm Y}$ is isomorphic to the principal open subset
 ${\rm D}(f)$ of ${\rm
   Spec}~(k[t_1,\ldots, t_n])$, where $$f = \prod_{i=1}^{n} \ \prod_{(a_i,\ldots, a_n) \in
 k^{n-i+1}}
 (1 + a_i t_i + a_{i+1} t_i t_{i+1} + \ldots + a_n t_i \ldots t_n).$$
 
In particular, ${\rm W}_{\rm Z}$ is affine. Moreover, the intersection of the exceptional divisor of $p$ with the open affine set
${\rm W}_{\rm Z}$ coincides with ${\rm div}(t_1 \cdots t_n)$, hence
has simple normal crossings. Using the isomorphism between ${\rm
  U}_{\rm Z}$ and ${\rm W}_{\rm Z}$ induced by $f$, we deduce that
${\rm U}_{\rm Z}$ is affine and that ${\rm D} \cap {\rm U}_{\rm Z}$ has simple
normal crossings. Since the sets ${\rm
  U}_{\rm Z}$ for all choices of complete flags form an open affine
covering of ${\rm X}$, the divisor ${\rm D}$ has simple normal
crossings on ${\rm X}$. 

We now claim that the intersection $\Sigma$ of any familly of $d$ irreducible
components of ${\rm D}$ is either empty or irreducible. Indeed, assume
that $\Sigma$ is non-empty and reducible. Non-emptyness
amounts to saying that these components correspond to linear subspaces in
some flag $\mathcal{F}$. Pick a complete flag $\mathcal{F}'$
containing $\mathcal{F}$. In the corresponding affine chart ${\rm
  U}_{\rm Z}$, the intersection of the $d$ components which we consider
is irreducible, hence there must be a component $\Sigma_0$ of $\Sigma$ which lies
in ${\rm X} - {\rm U}_{\rm Z}$. Since, by construction, ${\rm X} - {\rm
  U}_{\rm Z}$ is the union of some irreducible components of ${\rm D}$,
we see that $\Sigma_0$ must be contained in a $(d+1)$-th irreducible
component of ${\rm D}$. But this contradicts the normal crossing
property of ${\rm D}$. In view of the discussion before Lemma 1.3,
this shows that the strata of ${\rm D}$ are in one-to-one
correspondence with flags of linear subspaces.

If we start with a stratum ${\rm Z}$ corresponding to a partial flag $\mathcal{F}$, the set $ {\rm U}_{\rm Z} = {\rm X} - \bigcup_{{\rm L} \notin
    \mathcal{F}} {\rm E}_{\rm L}$ is the intersection of all ${\rm
  U}_{\Z'}$ for strata ${\rm Z}'$ corresponding to complete flags containing $\mathcal{F}$. Hence it is open affine as a finite intersection of open affines in a separated $k$-scheme. \end{proof}
  
In order to extend an automorphism of $\Omega(\V)$ to first $\X$ and
then to $\mathbf{P}(\V)$, we look at its action on the discrete valuations associated to the components of $\D$. For each ${\rm L} \in \mathcal{L}({\rm V})$, the local ring at the generic point of the hypersurface ${\rm E}_{\rm L}$ is a discrete valuation ring in the function field $\kappa({\rm V})$ of $\X$. We denote by ${\rm ord}_{\rm L}$ the corresponding discrete valuation on $\kappa({\rm V})$, and we write 
\[\Gamma({\rm V}) = \{{\rm ord}_{\rm L}: {\rm L} \in
  \mathcal{L}({\rm V}) \}\]
  for the set of all these valuations. Note that  $\kappa({\rm V})$ is the function field of both $\mathbf{P}({\rm V})$ and $\Omega(\V)$. If ${\rm L}$ is a \emph{hyperplane} in $\mathbf{P}({\rm V})$, then the valuation ${\rm ord}_\LL$ is the one given by the local ring of $\mathbf{P}(\V)$ at the generic point of ${\rm L}$. 

\vskip2mm The sets $\mathcal{L}({\rm V})$ and $\Gamma({\rm V})$ come with a natural
simplicial structure, for which the $q$-simplices correspond to flags
of linear subspaces of length $q-1$.

\begin{prop}\label{prop-valuations} Let $\varphi$ be a $k$-automorphism of \ $\Omega({\rm V})$ and let $\varphi^*$ be the induced automorphism of the set of valuations on the function field $\kappa({\rm V})$.

\begin{itemize}
\item[(i)] The birational map $\varphi$ extends to a $k$-automorphism of \ ${\rm X}$ if and only
  if $\varphi^*$ preserves the set $\Gamma({\rm V})$ and its simplicial structure.

\item[(ii)] The birational map extends to a $k$-automorphism of \ $\mathbf{P}({\rm
V})$ if and only if $\varphi^*$
preserves the subset of \ $\Gamma({\rm V})$ defined by hyperplanes.
\end{itemize}
\end{prop}

\begin{proof} (i) The condition is necessary because the simplicial
set $\Gamma({\rm V})$ describes the incidence relations between
irreducible components of ${\rm D}$ (Lemma \ref{lemma-blow-up}). To see that
it is sufficient, we use the covering of ${\rm X}$ by the open affine
subsets $${\rm U}_{\rm Z} =  {\rm X} - \bigcup_{{\rm L} \notin
  \mathcal{F}} {\rm E}_{\rm L}$$
where ${\rm Z}$ denotes a stratum of ${\rm D}$ and $\mathcal{F}$ is
the corresponding flag of linear subspaces of $\mathbf{P}({\rm
  V})$. If $\varphi$ preserves $\Gamma({\rm V})$ with its simplicial
structure, then there exists for every stratum ${\rm Z}$ another stratum
${\rm Z}'$ such that the rational map $${\rm U}_{{\rm Z}'}
\dashrightarrow {\rm U}_{\rm Z}$$
induced by $\varphi$ is defined at each point of height $1$.

Since
${\rm U}_{\rm Z}$ is affine and ${\rm U}_{{\rm Z}'}$ is noetherian and
normal, this rational map is everywhere defined on ${\rm U}_{{\rm
    Z}'}$ \cite[20.4.12]{ega4} and therefore $\varphi$ extends to an
automorphism from ${\rm X}$ to ${\rm X}$ (apply this argument to $\varphi^{-1}$).

(ii) If the morphism $\varphi:\Omega(\V) \rightarrow \Omega(\V)$ preserves all
valuations ${\rm ord}_\LL$ coming from hyperplanes, then for every
hyperplane $\LL$ in $\pv$ there exists a hyperplane $\LL'$ such that
the rational map $$\pv - \LL' \dashrightarrow \pv - \LL$$ induced by
$\varphi$ is defined at every point of height $1$, and the conclusion follows as for (i). \end{proof}

\section{Step 1 -- Valuations and analytic geometry}

This section is devoted to the first step toward the theorem, namely
the fact that every $k$-automorphism of $\Omega(\V)$ extends to a
$k$-automorphism of ${\rm X}$. 

\begin{prop}
\label{prop-analytic}
Let ${\rm Aut}_k({\rm X},{\rm D})$ denote the group of $k$-automorphisms of \ ${\rm X}$ which preserve ${\rm D}$.
The canonical map
$${\rm Aut}_k({\rm X},{\rm D}) \rightarrow {\rm Aut}_k\bigl(\Omega(\V)\bigr), \ \ \varphi \mapsto \varphi_{|\Omega(\V)}$$
is an isomorphism.
Equivalently, every $k$-automorphism of $\Omega(\V)$ extends to a $k$-automorphism of \ ${\rm X}$.
\end{prop}

We can study this problem from a nice geometric viewpoint in the framework of Berkovich spaces.

Endowed with the trivial absolute value, $k$ becomes a complete non-Archimedean field. There is a well-defined category of $k$-analytic spaces, together with an analytification functor ${\rm Z} \leadsto {\rm Z}^{\rm an}$ from the category of $k$-schemes locally of finite type.
If ${\rm Z}$ is affine, then the topological space underlying ${\rm Z}^{\rm an}$ is the set of multiplicative $k$-seminorms on $\mathcal{O}({\rm Z})$ with the topology generated by
evaluation maps $x \mapsto |f(x)|:=x(f)$, where $f\in~\mathcal{O}({\rm
  Z})$. Imposing the additional condition that all seminorms are bounded by $1$ on
the algebra $\mathcal{O}({\rm Z})$, we obtain a compact domain $\mathrm{Z}^\beth$ in
${\rm Z}^{\rm an}$ equipped with a \emph{specialization} map ${\rm sp} :
{\rm Z}^\beth \rightarrow {\rm Z}$ (denoted by $r$ in
\cite{Thuillier}) which sends a multiplicative seminorm $x$ to the
prime ideal $\{f \in \mathcal{O}({\rm Z}) \ | \ |f(x)|<1\}$. The reader is refered to \cite[Section 3.5]{Berkovich_Book} and
\cite[Section 1]{Thuillier} for a detailed account.

Working in the analytic category over $k$ allows us to realize $\Gamma({\rm V})$ as a set of rays in $\Omega(\V)^{\rm an}$: for each ${\rm L} \in \mathcal{L}({\rm V})$, the map
$$\varepsilon_{\rm L} : (0,1] \rightarrow \Omega(\V)^{\rm an}, \ \ r \mapsto r^{{\rm ord}_{\rm L}( \cdot )}$$
is an embedding and $\varepsilon_{\rm L}(1)$ is the canonical point of $\Omega(\V)^{\rm an}$, namely the point corresponding to the trivial absolute value on $\kappa({\rm V})$.
Now, the proposition will follow from the fact that this collection of rays is the $1$-skeleton of a \emph{conical complex}~$\mathfrak{S}({\rm V})$ in $\Omega(\V)^{\rm an}$ which is preserved by every $k$-automorphism of $\Omega(\V)$.

\vskip2mm
This conical complex $\mathfrak{S}({\rm V})$ is the \emph{fan}
$\mathfrak{S}_0({\rm X},{\rm D})$ of the toroidal embedding
$\Omega(\V) \hookrightarrow {\rm X}$ introduced in \cite[Section 3.1
  and Proposition 4.7]{Thuillier}, following
\cite{Berkovich_Smooth}. Let us describe this construction in the particular case we consider here. 

\vskip1mm \begin{itemize}
\item[(a)] The canonical
map $$r : {\bf A}^{n,\beth}_k \rightarrow [0,1]^n, \ \ \ x \mapsto
(|t_1(x)|,\ldots, |t_n(x)|)$$
has a continuous section $j$ defined by mapping a tuple $r \in
[0,1]^n$ to the following diagonalizable multiplicative seminorm on
$k[t_1,\ldots, t_n]$ :
$$\sum_{\nu \in {\bf N}^n} a_{\nu}
  t^\nu \mapsto \max_{\nu} |a_\nu| r_1^{\nu_1} \cdots r_n^{\nu_n}.$$ 
\item[(b)] Let ${\rm D}(t_1,\ldots,t_n)$ denote the invertibility locus of $t_1,\ldots,t_n$. Intersecting the image of $j$ with the open domain ${\rm
    D}(t_1,\ldots,t_n)^\beth$, we obtain a closed subset
${\rm C}_n \subset {\rm D}(t_1,\ldots,t_n)^\beth$ homeomorphic to
  the cone $(0,1]^n$. The map $\tau = r \circ j$ is a retraction of
    ${\rm D}(t_1,\ldots,t_n)^\beth$ onto ${\rm C}_n$. Its fiber
    over a point $x \in {\rm C}_n$ is a $k$-affinoid domain whose
    Shilov boundary is reduced to $\{x\}$.
\item[(c)] We identify ${\rm C}_n$ and $(0,1]^n$ via $r$. For ${\rm I} \subset \{1,\ldots, n\}$, let ${\rm C}_n^{\rm
  I}$ denote the face of ${\rm C}_n$ defined by $r_i=0$ for every
$i \in {\rm I}$. The specialization map ${\rm sp} : {\rm
  D}(t_1,\ldots,t_n)^\beth \rightarrow {\bf A}^n_k$ sends the interior
of ${\rm C}_n^{\rm I}$ to the generic point of the locally closed
subscheme ${\rm Z}_{\rm I} = {\rm V}(t_i, \ i \in {\rm I}) \cap {\rm D}(t_j, \ j \notin
{\rm I})$. This implies that ${\rm C}_n^{\rm I}$ is contained in ${\rm
  U}^\beth = {\rm sp}^{-1}({\rm U})$ for any open neighborhood ${\rm U}$
of the generic point of ${\rm Z}_{\rm I}$.
\item[(d)] We can also recover the monoid $r_1^{\bf N} \cdots r_n^{\bf
  N}$ defining the integral affine structure on $(0,1]^n$ from the analytic structure of ${\bf
      A}^n_k$. Indeed, this is precisely the monoid of functions
    $|f| : {\rm C}_n \rightarrow (0,1]$ induced by germs $f \in
        \mathcal{O}_{{\bf A}^n_k,0}$ invertible on ${\rm D}(t_1,
        \ldots,t_n)$. Similarly, the submonoid corresponding to the
        face ${\rm C}_n^{\rm I}$ comes from germs of
        $\mathcal{O}_{{\bf A}^n_k}$ at the generic
        points of ${\rm Z}_{\rm I}$ which are invertible over ${\rm
          D}(t_1,\ldots, t_n)$.
\item[(e)] We now return to the scheme ${\rm X}$ with its simple
  normal crossing divisor ${\rm D}$. Fix a stratum
  ${\rm Z}$ with generic point $\eta_{\rm Z}$ and let $\Lambda_{\rm
    Z}^+$ denote the mono\"{\i}d of germs in $\mathcal{O}_{{\rm X},\eta_{\rm
      Z}}$ whose restriction to $\Omega({\rm V})$ is invertible. As in the
  proof of Lemma \ref{lemma-blow-up}, there is an open immersion
  $(t_1,\ldots, t_n) : {\rm U}_{\rm Z} \rightarrow {\bf A}^n_k$ identifying ${\rm Z}$ with a
  non-empty open subset of ${\rm Z}_{\rm I}$ for a suitable subset
  ${\rm I}$ of $\{1,\ldots, n\}$. By transport of structure, we obtain
  a closed subset ${\rm C}_{\rm Z}$ of ${\rm U}_{\rm Z}^\beth -
  {\rm Z}^\beth$ such that the natural map $${\rm C}_{\rm Z}
  \rightarrow {\rm Hom}_{\bf Mon}(\Lambda_{\rm Z}^+/k^\times,(0,1]),
    \ \ \ x \mapsto (f \mapsto |f(x)|)$$ is a homeomorphism. Covering
    ${\rm X}$ by the open subschemes ${\rm U}_{\rm Z}$, we can glue
    the cones ${\rm C}_{\rm Z}$ along common faces in $\Omega({\rm
      V})^{\rm an}$ to define a cone complex $\mathfrak{S}({\rm
      V})$. This gluing is compatible with local retractions, so we get
    a retraction of $\Omega({\rm V})^{\rm an}$ onto $\mathfrak{S}({\rm V})$.
\end{itemize}

\vskip2mm The following propery of the conical complex
$\mathfrak{S}({\rm V})$ is specific to our situation and is the key
point to prove Proposition \ref{prop-analytic}. It may be interesting
to look for other ``natural'' toroidal compactifications satisfying
this condition.

\begin{lemma}
\label{lemma-fan} The map
$$\iota : \mathfrak{S}({\rm V}) \rightarrow {\rm Hom}_{\mathbf{Ab}}\Bigl(\mathcal{O}\bigl(\Omega(\V)\bigr)^{\times}/k^\times, \mathbf{R}_{>0}\Bigr), \ x \mapsto (f \mapsto |f(x)|)$$
is a closed embedding inducing the integral affine structure on each cone.
Moreover, (the images of) distinct cones span distinct linear spaces.
\end{lemma}

\begin{proof} Roughly speaking, this statement means that there are enough invertible functions on $\Omega(\V)$.
Consider a stratum ${\rm Z}$ of ${\rm D}$ corresponding to a flag $\mathcal{F}$ of non-trivial linear subspaces of $\mathbf{P}({\rm V})$ and pick a basis $(e_0, \ldots, e_n)$ of ${\rm V}$ such that $\mathcal{F}$ is a subflag of
$${\rm Z}(e_1,\ldots, e_n) \subset {\rm Z}(e_2,\ldots, e_n) \subset \ldots \subset {\rm Z}(e_n).$$
The explicit description of ${\rm X}$ given at the end of the proof of Lemma \ref{lemma-blow-up} shows that the tuple $(e_1/e_0,e_2/e_1,\ldots,e_n/e_{n-1})$ of elements in $\mathcal{O}_{\rm X,\eta_{\rm Z}}$ contains a regular system of parameters defining ${\rm D}$ at $\eta_{\rm Z}$. Therefore, the map $\iota$ induces an integral affine embedding of the cone ${\rm C}_{\rm Z}$.

\vskip2mm Furthermore, we claim that the following fact is true:
\emph{given two distinct cones ${\rm C}, {\rm C}'$, there exists $f \in \mathcal{O}\bigl(\Omega(\V)\bigr)^\times$ such that $|f| = 1$ on one of them and $|f|<1$ on the interior of the other}.
Injectivity of the map $\iota$ and the last statement of the Lemma follow immediately. 

\vskip2mm  We finish the proof by establishing the claim. Given two non-zero vectors $v,v' \in {\rm V}$ and a non-trivial linear subspace ${\rm L} \subset
\mathbf{P}({\rm V})$, the function $v/v'$ is either a unit, a
uniformizer or the inverse of a uniformizer at the generic point of
${\rm E}_{\rm L}$, according to the position of ${\rm Z}(v)$ and ${\rm
  Z}(v')$ with respect to ${\rm L}$. It follows that 
\begin{itemize}
\item[(a)] $|v/v'|<1$ on $\varepsilon_{\rm L}(0,1)$, if ${\rm L}
  \subset {\rm Z}(v)$ and ${\rm L} \not\subset {\rm Z}(v')$
\item[(b)] $|v/v'|>1$ on $\varepsilon_{\rm L}(0,1)$, if ${\rm L}
  \subset {\rm Z}(v')$ and ${\rm L} \not\subset {\rm Z}(v)$
\item[(c)] $|v/v'|=1$ on $\varepsilon_{\rm L}(0,1]$, if the
  hyperplanes ${\rm Z}(v)$ and ${\rm Z}(v')$ are in the same position with respect to~${\rm L}$.
\end{itemize}
Consider two distinct strata ${\rm Z}$, ${\rm Z}'$ of ${\rm D}$, corresponding to distinct flags $\mathcal{F}, \mathcal{F}'$ of non-trivial linear subspaces.
Pick a linear space ${\rm L}$ occurring in only one of them, say $\mathcal{F}$, and set $i = {\rm dim}~{\rm L}$. We embed $\mathcal{F}'$ into a complete flag $({\rm L}_0 \subset {\rm L}_1 \subset \ldots \subset {\rm L}_{n-1})$ such that ${\rm L}_i \neq {\rm L}$.

We claim that this assumption guarantees the existence of two hyperplanes ${\rm H}, {\rm H}'$ such that 

\begin{itemize}
\item[-] ${\rm L} \subset {\rm H}$ and ${\rm L}_i \not\subset {\rm H}$
\item[-] ${\rm L}_i \cap {\rm H} = {\rm L}_i \cap {\rm H}'$ and ${\rm L} \not\subset {\rm H}'$.
\end{itemize}

In order to prove this claim, we argue with the corresponding linear
quotient spaces of $\V$. Let ${\rm L} = \mathbf{P}(\V/\mathrm{W})$ and
${\rm L}_i = \mathbf{P}(\V/\mathrm{W}_i)$ where $ \mathrm{W}$ and $ \mathrm{W}_i$ are different linear subspaces of $\V$ of dimension $n-i$. Choose a vector $u \in  \mathrm{W}$ which is not contained in $ \mathrm{W}_i$, and a vector $u_i \in  \mathrm{W}_i$ which is not contained in $ \mathrm{W}$. We denote by $\mathrm{U}$ the line in $\V$ generated by $u$ and by $\mathrm{U}'$ the line generated by $u' = u + u_i$. The corresponding hyperplanes ${\mathrm H} = \mathbf{P}(\V/\mathrm{U})$ and $\mathrm{H}' = \mathbf{P}(\V/\mathrm{U}')$ have the desired properties. 

\vskip2mm
In particular, ${\rm H}$ and ${\rm H}'$ are in the same position with respect to ${\rm L}_0, \ldots, {\rm L}_{n-1}$.
Given any equations $v,v' \in {\rm V}$ of ${\rm H}$ and ${\rm H}'$ respectively, we thus obtain $|v/v'|=1$ on ${\rm C}_{\rm Z'}$.
Let us now consider the flag $\mathcal{F}$. Any linear subspace
${\rm M} \in \mathcal{F}$ contained in ${\rm H}'$ is
necessarily contained in ${\rm L}$, hence in ${\rm H}$, therefore $|v/v'| \leqslant 1$ on the
ray $\varepsilon_{{\rm M}}(0,1]$. Since $|v/v'|<1$ on the interior of
 the ray $\varepsilon_{\rm L}(0,1]$, we deduce that $|v/v'|<1$ on the
  interior of the cone ${\rm C}_{\rm Z}$. \end{proof}

\medskip 

\begin{proof}[Proof of proposition \ref{prop-analytic}] First, we observe that $\mathfrak{S}({\rm V})$ coincides with the set $\Omega(\V)^{\rm an}_{\rm max}$ of maximal points of $\Omega(\V)^{\rm an}$ for the following ordering:
$$x \preccurlyeq y \ \ \ \Longleftrightarrow \ \ \ \forall f \in \mathcal{O}(\Omega(\V)^{\rm an}), \ \ |f(x)| \leqslant |f(y)|.$$
For any point $x \in \Omega({\rm V})^{\rm an}$, we have $x
\preccurlyeq \tau(x)$ because the fiber $\tau^{-1}(\tau(x))$ is a
$k$-affinoid domain with Shilov boundary $\{\tau(x)\}$. This implies
the inclusion $\Omega(\V)^{\rm an}_{\rm max} \subset \mathfrak{S}({\rm
  V})$.

We apply Lemma \ref{lemma-fan} to get the converse inclusion. If a point $x
\in \mathfrak{S}({\rm V})$ is dominated by a point $x' \in \Omega({\rm
  V})^{\rm an}$, then it is also dominated by $\tau(x')$. However,
for any two distinct points $x,y$ in $\mathfrak{S}({\rm V})$, there exists $f \in \mathcal{O} (\Omega(\V)^{\rm an})^\times$ such that $|f(x)| \neq |f(y)|$, hence such that $|f(x)| < |f(y)|$ and $|({1 \over f})(x)|>|({1 \over f})(y)|$ 
or vice versa, and therefore $x$ and $y$ are incomparable. In
particular, we get $x = \tau(x')$ and thus $x$ is maximal.

\vskip2mm
The above characterization of $\mathfrak{S}({\rm V})$ as a closed subset of $\Omega(\V)^{\rm an}$ implies that it is preserved by any $k$-automorphism $\varphi$ of $\Omega(\V)$.
It remains to check that the homeomorphism of $\mathfrak{S}({\rm V})$ induced by $\varphi$ also preserves the conical structure.
Let $\Phi$ denote the linear automorphism of ${\rm Hom}_{\mathbf{Ab}}(\mathcal{O}(\Omega(\V)^{\rm an})^\times, \mathbf{R}_{>0})$ deduced from $\varphi$. Given an 
$n$-dimensional cone ${\rm C} \subset \mathfrak{S}({\rm V})$, the
image of its interior is disjoint from the $(n-1)$-skeleton of
$\mathfrak{S}({\rm V})$; otherwise it would meet the interiors of two
distinct $n$-dimensional cones ${\rm C}'$, ${\rm C}''$, hence $\langle
\iota~{\rm C}'\rangle = \Phi(\langle\iota~{\rm C}\rangle) = \langle
  \iota~{\rm C}''\rangle$ contradicting Lemma \ref{lemma-fan}. It follows that if $\varphi({\rm C})$ is contained in some $n$-dimensional cone ${\rm
  C}'$, and thus $\varphi({\rm C}) = {\rm C}'$ by considering $\varphi^{-1}$.
The assertion for lower dimension cones follows at once by considering
faces since the automorphism $\Phi$ is linear.

In particular, we see that $\varphi$ preserves the $1$-skeleton of
$\mathfrak{S}({\rm V})$, hence the set $\Gamma({\rm V})$ of discrete
valuations on $\kappa({\rm V})$ associated with irreducible components of ${\rm D} = {\rm X} -
\Omega(\V)$, together with the simplicial structure reflecting 
the incidence relations between these components. By Proposition
\ref{prop-valuations} (i), this implies that $\varphi$ extends to a
$k$-automorphism of ${\rm X}$. \end{proof}

\begin{remark}\label{remark-complex}
\emph{\begin{enumerate}
\item Let ${\rm D}$ be a simple normal crossing divisor on a smooth and proper (connected) scheme ${\rm X}$ over $k$.
Even if $\Omega(\V) = {\rm X}-{\rm D}$ is affine, Lemma \ref{lemma-fan} and its consequences may fail.
For example, consider the case ${\rm X} = \mathbf{P}^n_k$. If ${\rm D}$ is a hyperplane, then $\mathfrak{S}_0({\rm X},{\rm D})$ is a $1$-dimensional cone whereas $\Omega(\V)^{\rm an}_{\rm max}$ is empty.
If ${\rm D}$ is the union of the coordinate hyperplanes, then $\Omega(\V) = \mathbf{G}_{\rm m}^n$ and $\mathfrak{S}_0({\rm X},{\rm D}) = \Omega(\V)^{\rm an}_{\rm max}$ is the toric fan, but the map $\iota$ is bijective, hence all maximal cones span the same linear space.
In fact, the inversion $(t_1,\ldots, t_n) \mapsto (t_1^{-1},\ldots, t_n^{-1})$ on $\mathbf{G}_{\rm m}^n$ transforms the fan $\mathfrak{S}_0({\rm X},{\rm D})$ into its opposite, hence does not preserve the conical structure.
This reflects the fact that this automorphism of $\mathbf{G}_{\rm m}^n$ does not extend to $\mathbf{P}^n$.
\item The conical complex $\mathfrak{S}({\rm V})$ is also the \emph{vectorial building} of ${\rm PGL}({\rm V})$, but this is somehow fortuitous and irrelevent from the viewpoint of automorphisms.
In general, there exists for any connected and split semi-simple
$k$-group ${\rm G}$ a canonical embedding of the vectorial building
$\mathcal{V}({\rm G},k)$ of ${\rm G}(k)$ into the analytification of
an open affine subscheme $\Omega$ in any flag variety ${\rm Y}$ of ${\rm G}$
\cite[Section 5.5]{Berkovich_Book}. However, this observation does not lead
to a generalization of Theorem \ref{main-theorem}, at least along the
lines of the present proof. Indeed, while we made crucial use of the
fact that $\mathfrak{S}({\rm V})$ is the fan of a normal crossing
divisor, we doubt that $\mathcal{V}({\rm G},k)$ can be realized as
the fan of a toroidal compactification of $\Omega(\V)$ if $({\rm G}', {\rm
  Y}) \neq \bigl({\rm PGL}({\rm V}), \mathbf{P}({\rm V})\bigr),
\bigl({\rm PGL}({\rm V}), \mathbf{P}({\rm V}^\vee)\bigr)$. 
\item It may be interesting to try to extend our method, based on the
  study of toroidal compactifications, to determine the automorphism
  groups of other period domains. 
\item Whether the above proposition can be proved without analytic geometry is not clear.
\end{enumerate}}
\end{remark}

\section{Step 2 -- Geometry of the blow-up}
The second step in the proof of the theorem relies on elementary intersection theory on ${\rm X}$, which we review in this section. The standard reference is \cite{Fulton}.

\vskip3mm The Chow ring ${\rm CH}^*$ is a contravariant functor from the category of smooth $k$-schemes to the category of graded commutative rings.
For any smooth $k$-scheme ${\rm X}$, the abelian group underlying ${\rm CH}^*({\rm X})$ is the free abelian group on integral subschemes of ${\rm X}$ modulo rational equivalence, and it is graded by codimension.
Multiplication comes from the intersection product.
We write $[{\rm Z}]$ for the class of a closed subscheme ${\rm Z}$ of ${\rm X}$. 

\vskip2mm We are going to use the following two basic facts.

\begin{itemize}
\item[(a)] Let ${\rm Y}$ be a regularly embedded closed subscheme of
  ${\rm X}$ of codimension $\geqslant 2$ and let $\pi : \widetilde{\rm X} \rightarrow {\rm X}$ be the blow-up of ${\rm X}$ along ${\rm Y}$, with exceptional divisor $\widetilde{\rm Y}$.
The canonical map
$${\rm CH}^1({\rm X}) \oplus \mathbf{Z}[\widetilde{\rm Y}] \rightarrow {\rm CH}^1(\widetilde{\rm X}), \ \ (z, n[\widetilde{\rm Y}]) \mapsto \pi^*(z) + n[\widetilde{\rm Y}]$$
is an isomorphism \cite[Proposition 6.7]{Fulton}. 
\item[(b)] In the situation of (a), let ${\rm V}$ be an integral subscheme of ${\rm X}$ with strict transform $\widetilde{\rm V}$.
If ${\rm codim}({\rm Y},{\rm X}) \leqslant {\rm codim}({\rm V}\cap {\rm Y} ,{\rm V}) $, then
$$\pi^*[{\rm V}] = [\widetilde{\rm V}]$$ in ${\rm CH}^*(\widetilde{\rm X})$ \cite[Corollary 6.7.2]{Fulton}. 
\end{itemize}

Now we focus on the particular case where $\pi : {\rm X} \rightarrow \mathbf{P}({\rm V})$ is the blow-up along the full hyperplane arrangement, with exceptional divisor ${\rm D}$.
 
\begin{lemma}
\label{lemma-chow}
We have
$$\mathrm{CH}^1({\rm X}) = \mathbf{Z}h \oplus \bigoplus_{\rm L} \mathbf{Z}[{\rm E}_{\rm L}],$$
where $h = \pi^{*}[{\rm H}]$ denotes the pull-back of the hyperplane class $[{\rm H}]$ on $\mathbf{P}({\rm V})$ and ${\rm L}$ runs over the set of non-trivial linear subspaces of \ $\mathbf{P}({\rm V})$ of codimension at least $2$.
\end{lemma}

\begin{proof} For any non-trivial linear subspace ${\rm L}$ of
$\mathbf{P}({\rm V})$ of dimension $i \in \{0,\ldots, n-1\}$, let
${\rm E}_{\rm L}^{(i)} \subset {\rm X}_i$ denote the blow-up of its strict transform in
${\rm X}_{i-1}$; this is a smooth irreducible hypersurface.
Recall that we have $\pi = \pi_0 \circ \pi_1 \circ \ldots \circ \pi_{n-1}$, where $\pi_{n-1}$ is an isomorphism. Applying (a) iteratively to each blow-up $\pi_0, \ldots, \pi_{n-2}$, we obtain that ${\rm CH}^1({\rm X})$ is the free abelian group on $h$ and the classes $(\pi_{i+1} \circ \ldots \circ \pi_{n-1})^*[{\rm E}_{\rm L}^{(i)}]$, where $i \in \{0, \ldots, n-2\}$ and ${\rm L}$ runs over the set of $i$-dimensional linear subspaces of $\mathbf{P}({\rm V})$. 

\vskip1mm The conclusion follows from the additional fact that we have an equality
$$(\pi_{i+1} \circ \ldots \circ \pi_{n-1})^*[{\rm E}_{\rm L}^{(i)}] = [{\rm E}_{\rm L}]$$
in ${\rm CH}^1({\rm X})$ for any linear subspace ${\rm L}$ of dimension $i \in \{0,\ldots, n-2\}$. This is an immediate consequence of (b), since the center of each blow-up $\pi_j$, with $j \in \{i+1, \ldots, n-1\}$, is transversal to the strict transform of ${\rm E}_{\rm L}^{(i)}$ in ${\rm X}_{j-1}$. 
\end{proof}

\medskip

For each integer $d \geqslant 1$, we define
$$\lambda(d) = \#~\left\{\begin{array}{c} {\rm non-trivial \ linear \ subspaces} \\ {\rm of \ codimension} \ \geqslant 2 \quad {\rm in} \ \mathbf{P}_k^d\end{array} \right\}.$$
Additionally, we set $\lambda(0)=0$.

\begin{lemma}
\label{lemma-rank}
Let ${\rm L} \subset \mathbf{P}({\rm V})$ be a non-trivial linear subspace of dimension $d$; note that $d \in \{0, \ldots, n-1\}$.
\begin{itemize}
\item[(i)] We have $${\rm rk}~{\rm CH}^1({\rm E}_{\rm L}) = \lambda(d)
  + \lambda(n-1-d) + \varepsilon(d),$$ where
      $\varepsilon(d) = 1$ if $d \in \{0,n-1\}$ and $\varepsilon(d) =
      2$ otherwise.
\item[(ii)] For every linear subspace ${\rm L}' \subset \mathbf{P}({\rm
  V})$ of dimension $d'$ satisfying $d < d' < n-1-d$, the following
  inequality holds $${\rm
  rk}~ {\rm CH}^1({\rm E}_{\rm L}) > {\rm rk}~ {\rm CH}^1({\rm E}_{\rm L'}).$$
\end{itemize}
\end{lemma}

\begin{proof} (i) Let ${\rm L}_{d-1}$
(resp. $\widetilde{\rm L})$ denote the strict transform of ${\rm L}$
in ${\rm X}_{d-1}$ (resp. in ${\rm X}_d$). The scheme ${\rm E}_{\rm
  L}$ is the blow-up of $\widetilde{\rm L}$ along the hypersurface
arrangement induced by hyperplanes of $\mathbf{P}({\rm V})$ containing
${\rm L}$. Applying (a), we obtain \begin{eqnarray*}{\rm rk}~{\rm
    CH}^1({\rm E}_{\rm L}) & = & {\rm rk}~{\rm CH}^1(\widetilde{\rm
    L}) + \# \left\{\begin{array}{c} {\rm linear \ spaces \ of
    \ codim} \ \geqslant 2 \\  {\rm strictly \ containing} \ \ {\rm
    L} \end{array}\right\} \\ & = &  {\rm rk}~{\rm
    CH}^1(\widetilde{\rm L}) + \lambda(n-d-1).\end{eqnarray*} Since
$\widetilde{\rm L} = \mathbf{P}(\mathcal{N})$, where $\mathcal{N}$ is
the conormal sheaf to ${\rm L}_{d-1}$ in ${\rm X}_{d-1}$, of rank
$n-d$, it follows from \cite[Theorem 3.3, (b)]{Fulton} that $${\rm rk}~{\rm CH}^1(\widetilde{\rm L}) = {\rm rk}~{\rm CH}^0({\rm L}_{d-1}) + {\rm rk}~{\rm CH}^1({\rm L}_{d-1}) = 1 + {\rm rk}~{\rm CH}^1({\rm L}_{d-1})$$ if $0 \leqslant d < n-1$, and $${\rm rk}~{\rm CH}^1(\widetilde{\rm L}) = {\rm rk}~{\rm CH}^1({\rm L}_{d-1})$$ if $d=n-1$.

\vskip1mm Finally, since ${\rm L}_{d-1}$ is the blow-up of ${\rm L}$ along the full hyperplane arrangement, $${\rm rk}~{\rm CH}^1({\rm L}_{d-1}) = {\rm rk}~{\rm CH}^1({\rm L}) + \# \left\{\begin{array}{c} {\rm non-trivial \ linear \ subspaces} \\  {\rm of \ codimension \ } \geqslant 2 \ {\rm in \ L} \end{array} \right\},$$ hence $${\rm rk}~{\rm CH}^1({\rm L}_{d-1}) = \left\{\begin{array}{ll} 1 + \lambda(d) & {\rm if} \ 0<d\leqslant n-1 \\ 0 & {\rm if} \ d=0. \end{array} \right.$$

\vskip2mm (ii) In view of (i), it is enough to prove the
inequality $$(1) \quad \quad \lambda(n-1-d) - \lambda(n-1-d') > \lambda(d') -
\lambda(d) + 1$$ for any $d,d' \in \{0,\ldots, n-1\}$ such that
$d<d'<n-1-d$.  

Let us first show that this statement follows from the inequality 
$$(2) \quad \quad \lambda(t)-\lambda(t-1) >\lambda(t-1)+1 \mbox{ for all } t \geqslant 2.$$
Indeed, assuming $(2)$, fix $d \in \{0,\ldots, n-1\}$ and $d'$ satisfying $d<d'<n-1-d$. Since $d' \geqslant d+1$, we have $n-1-d' \leqslant n-1-(d+1)$, and since $\lambda$ is an increasing function this implies
\[\lambda(n-1-d) - \lambda(n-1-d') \geqslant \lambda(n-1-d) - \lambda(n-1-(d+1)).\]
Now put 
$t = n-1-d$. Note that $t \geqslant 2$. Then $t-1 = n-1-(d+1)$,
and $(2)$ implies 
\[\lambda(n-1-d) - \lambda(n-1-(d+1)) > \lambda(n-1-(d+1)) + 1.\]
 
As $d' < n-1-d$, we have  $d' \leqslant n-1-d-1 = n-1-(d+1)$. Once more we use the fact that $\lambda$ is a non-negative increasing function to deduce
\[\lambda(n-1-(d+1)) + 1 \geqslant \lambda(d')  +1\geqslant \lambda(d') - \lambda(d) +1.\]
Combining the previous  inequalities gives $(1)$. 

Therefore it remains to prove $(2)$. 
If we fix a hyperplane 
${\rm H}$ and count non-trivial linear subspaces of codimension $ \geqslant 2$ in
$\mathbf{P}^t_k$ taking into account their position with respect to
${\rm H}$ (transverse to ${\rm H}$, or of codimension $\geqslant 2$ or $ = 1$ in ${\rm H}$), we obtain for $t \geqslant 2$
$$\lambda(t) = \nu(t) + \lambda(t-1) + \#~\mathbf{P}^{t-1}(k) > \nu(t)
+ \lambda(t-1) + 1,$$ where $\nu(t)$ denotes the number of non-trivial
linear subspaces of codimension at least $2$ in $\mathbf{P}^t_k$ which
are not contained in ${\rm H}$. Hence, it is enough to prove the inequality $$\nu(t) \geqslant \lambda(t-1)$$ for every integer $t \geqslant 2$. But this is obvious: given a hyperplane $\mathbf{P}^{t-1}_k \subset \mathbf{P}^t_k$ and a rational point $p$ in the complement of $\mathbf{P}^{t-1}(k)$, the map ${\rm L} \mapsto \langle {\rm L}, p\rangle$ embeds the set of codimension $d$ linear subspaces of $\mathbf{P}^{t-1}_k$ into the set of codimension $d$ linear subspaces of $\mathbf{P}^t_k$ which are not contained in $\mathbf{P}^{t-1}_k$. \end{proof}

\vskip4mm
We can now prove Theorem \ref{main-theorem}.

\smallskip 

Let us first show part (i).
Every $k$-automorphism $\varphi$ of $\Omega(\V)$ extends to a $k$-automorphism $\widetilde{\varphi}$ of ${\rm X}$ by Proposition \ref{prop-analytic}.
Hence it induces a permutation $\hat{\varphi}$ of non-trivial linear subspaces of $\mathbf{P}({\rm V})$ defined by $\widetilde{\varphi}({\rm E}_{\rm L}) = {\rm  E}_{\hat{\varphi}({\rm L})}$. Note that $\hat{\varphi}$ preserves the simplicial structure of flags in $\mathcal{L}({\rm V})$ because $\widetilde{\varphi}$ preserves the simplicial structure of strata of the boundary divisor. By Proposition \ref{prop-valuations} (ii) it suffices to prove that $\hat{\varphi}$ preserves \emph{hyperplanes}.

\vskip3mm

We argue by induction on $n = {\rm dim}~{\rm V} - 1 \geqslant 1$.
For $n=1$, the result is obvious.
For $n=2$, it is enough to compare self-intersections of components of
${\rm D}$ to conclude: for a point $p$ and a line $\ell$,
$${\rm deg}~[{\rm E}_p]^2 = - 1 \ \ \ {\rm and} \ \ \ {\rm deg}~[{\rm E}_\ell]^2 = {\rm deg}~\left(h - \sum_{q \in \ell(k)} [{\rm E}_{q}] \right)^2 = 1 - \#~\ell(k) = - (\#~k),$$
thus $\hat{\varphi}$ maps a line to a line.

\vskip3mm
In general, for any rational hyperplane ${\rm H}$ of $\mathbf{P}({\rm V})$, it follows from Lemma \ref{lemma-rank} that $\hat{\varphi}({\rm H})$ is either a hyperplane or a rational point.
Let us now assume that $n$ is at least $3$ and that the theorem has been proved in lower dimension.
If $\hat{\varphi}({\rm H})$ is a rational point $p$, then $\widetilde{\varphi}$ induces a $k$-isomorphism $\bar{\varphi}$ between ${\rm E}_{\rm H}$ and ${\rm E}_p$ which maps the divisor ${\rm D}_{\rm H} = \bigcup_{{\rm L} \neq {\rm H}} {\rm E}_{\rm H} \cap {\rm E}_{\rm L}$ onto the divisor ${\rm D}_p = \bigcup_{{\rm L} \neq \{p\}} {\rm E}_p \cap {\rm E}_{\rm L}$. 

\vskip1mm Since ${\rm E}_{\rm H}$ (resp. ${\rm E}_p$) is the blow-up of ${\rm H}$ (resp. $\mathbf{P}({\rm T}_p^\vee)$, where ${\rm T}_p$ denotes the tangent space of $\mathbf{P}({\rm V})$ at $p$) along the full hyperplane arrangement, with exceptional divisor ${\rm D}_{\rm H}$ (resp. ${\rm D}_p$), the theorem in dimension $n-1$ implies that $\bar{\varphi}$ is induced by a $k$-isomorphism between ${\rm H}$ and $\mathbf{P}({\rm T}_p^\vee)$, hence maps the components of ${\rm D}_{\rm H}$ defined by rational points of ${\rm H}$ to components of ${\rm D}_p$ defined by rational points of $\mathbf{P}({\rm T}_p^\vee)$, which is to say by (rational) lines in $\mathbf{P}({\rm V})$ containing $p$.

\vskip1mm Let $q$ be a rational point of ${\rm H}$ and let $\ell$
denote the line in $\mathbf{P}({\rm V})$ such that $$\widetilde{\varphi}({\rm
  E}_{\rm H} \cap {\rm E}_q) = {\rm E}_p \cap {\rm E}_\ell.$$
The two hypersurfaces ${\rm E}_\ell$ and $\widetilde{\varphi}({\rm
  E}_q)$ have the same non-empty intersection with $\widetilde{\varphi}({\rm
  E}_{\rm H}) = {\rm E}_p$, so $$\widetilde{\varphi}({\rm E}_q) = {\rm
  E}_\ell$$ since ${\rm D}$ is a normal crossing divisor. By Lemma \ref{lemma-rank}, this implies $n=2$ while we assumed $n \geqslant
3$. 

Therefore, $\hat{\varphi}$ preserves the set of hyperplanes.

\vskip2mm \noindent {\bf Remark} --- Carlo Gasbarri suggested that it should be possible to prove that $\hat{\varphi}$ preserves hyperplanes by looking at the canonical divisor on ${\rm X}$, which is a fixed point of $\widetilde{\varphi}^*$ in ${\rm CH}^1({\rm X})$. We sketch a way to combine this idea with results of Section 3. Using the classical formula for the canonical divisor of a blow-up \cite[Exercice II.8.5]{Hartshorne}, we obtain \begin{equation}{\rm K}_{\rm X} = \pi^*{\rm K}_{{\bf P}({\rm V})} + \sum_{{\rm L} \in \mathcal{L}({\rm V})} ({\rm codim}~{\rm L} - 1) [{\rm E}_{\rm L}] = -(n+1)h + \sum_{i=0}^{n-2} (n-i-1) \sum_{{\rm L} \in \mathcal{L}^i({\rm V})} [{\rm E}_{\rm L}].\end{equation} 
Let $\Gamma$ denote the subgroup of ${\rm CH}^1({\rm X})$ spanned by $\{[{\rm E}_{\rm L}]\}_{{\rm codim}~{\rm L} \geqslant 2}$. For any hyperplane ${\rm H}$, we have : $$\widetilde{\varphi}^*h = \widetilde{\varphi}^*\left([{\rm E}_{\rm H}] + \sum_{{\rm L} \subsetneq {\rm H}} [{\rm E}_{\rm L}]\right) = [{\rm E}_{\hat{\varphi}^{-1}({\rm H})}] + \sum_{{\rm L} \subsetneq {\rm H}} [{\rm E}_{\hat{\varphi}^{-1}({\rm L})}].$$ Since $\hat{\varphi}^{-1}$ preserves the simplicial structure of $\mathcal{L}({\rm V})$, it maps the link of ${\rm H}$ to the link of ${\rm  W} = \hat{\varphi}^{-1}({\rm H})$, hence linear subspaces of ${\rm H}$ to linear subspaces of ${\bf P}({\rm V})$ contained in or containing ${\rm W}$.  Since there are $\# {\bf P}({\rm V}/{\rm W})(k)$ hyperplanes containing ${\rm W}$, we obtain $$\widetilde{\varphi}^*h \equiv \#{\bf P}({\rm V}/{\rm W})(k)h \ \ ({\rm mod} \ \Gamma).$$ In particular, ${\rm dim}~\hat{\varphi}^{-1}({\rm H})$ does not depend on the hyperplane ${\rm H}$. Together with Lemma \ref{lemma-rank}, this observation implies that $\hat{\varphi}$ either preserves hyperplanes or swaps hyperplanes and points. 

\vskip1mm Assume that $\hat{\varphi}$ swaps hyperplanes and points. Then $${\rm K}_{\rm X} \equiv \widetilde{\varphi}^*{\rm K}_{\rm X} \equiv -(n+1) \widetilde{\varphi}^*h + (n-1) \sum_{p \in {\bf P}({\rm V})(k)} \widetilde{\varphi}^*[{\rm E}_p] \equiv \left(-(n+1)\#{\bf P}^{n-1}(k) + (n-1)\#{\bf P}^n(k)\right)h$$ modulo $\Gamma$, so Equation (1) implies $$(n+1)(q^n-q) = (n-1)(q^{n+1}-1) $$ with $q = \#k$. This identity cannot hold if $n>1$ since it would imply $q|n-1$, hence $n \geqslant q+1 \geqslant 3$ and $(n+1)/(n-1) \leqslant 2$, whereas $(q^{n+1}-1)/(q^n-q) > q \geqslant 2$. Therefore, $\hat{\varphi}$ has to preserve hyperplanes.

\section{Extension of the ground field}
We now indicate how to prove the second part of Theorem
\ref{main-theorem}. For every field extension ${\rm K}/k$, the base
change $\Omega({\rm V})_{\rm K}$ of $\Omega({\rm V})$ coincides with
the complement in $\mathbf{P}({\rm V})_{\rm K}$ of all
$k$-\emph{rational} hyperplanes. Since blowing-up commutes with base change, the
${\rm K}$-scheme $\X_{\rm K} = \X \otimes_k {\rm K}$ can be obtained
by blowing up $\mathbf{P}({\rm V})_{\rm K}$ along the
arrangement of all $k$-rational hyperplanes. Moreover, every
irreducible components $ {\rm E}_{\rm L}$ of ${\rm D}$ is
geometrically irreducible, and its base change $({\rm E}_{\rm L})_{\rm
  K}$ is the irreducible component of ${\rm X}_{\rm K} - \Omega({\rm
  V})_{\rm K}$ corresponding to the $k$-rational linear subspace ${\rm L}_{\rm K}$ of $\mathbf{P}({\rm V})_{\rm K}$.

\vskip1mm Let us consider a ${\rm K}$-automorphism $\varphi$ of
$\Omega({\rm V})_{\rm K}$. One proves exactly as in Proposition
\ref{prop-valuations} that $\varphi$ extends to a ${\rm
  K}$-automorphism of ${\rm X}_{\rm K}$ (resp. of $\mathbf{P}({\rm
  V})_{\rm K}$) if and only if $\varphi$ preserves the simplicial set
$\Gamma({\rm V}_{\rm K})$ of discrete valuations on $\kappa({\rm
  V}_{\rm K})$ coming from irreducible components of ${\rm D}_{\rm K}$
(resp. preserves the subset of $\Gamma({\rm V}_{\rm K})$ corresponding
to hyperplanes). Once again, this condition is established via analytic
geometry over the field ${\rm K}$ endowed with the trivial absolute
value. The key point is Lemma \ref{lemma-fan}, which holds for the fan
$\mathfrak{S}({\rm V}_{\rm K})$ of the normal crossing divisor ${\rm
  D}_{\rm K}$ on ${\rm X}_{\rm K}$. The proof works verbatim, but one could
also argue that $\mathfrak{S}({\rm V}_{\rm K})$ coincides with the
inverse image of $\mathfrak{S}({\rm V})$ under the projection map $p :
{\rm X}^{\rm an}_{\rm K} \rightarrow {\rm X}^{\rm an}$, so the statement holds for $\mathfrak{S}({\rm V}_{\rm K})$ since it holds for
$\mathfrak{S}({\rm V})$. We then prove as above that $\varphi$ extends
to a ${\rm K}$-automorphim of ${\rm X}_{\rm K}$.

\vskip1mm Lemma \ref{lemma-chow} and Lemma \ref{lemma-rank} 
also apply to ${\rm X} \otimes_k {\rm K}$, when we replace ''linear subspaces''
by ''$k$-rational linear subspaces''. It follows that the permutation of $k$-rational linear subspaces induced by
$\widetilde{\varphi}$ perserves the hyperplanes, hence $\varphi$
induces a ${\rm K}$-automorphism of $\mathbf{P}(\V)_{\rm K}$. This
automorphism preserves the set of $k$-rational hyperplanes. Pick a
$k$-basis of ${\rm V}$ and consider the corresponding coordinate
hyperplanes; since they are mapped to $k$-rational hyperplanes,
$\varphi$ is induced by a $k$-automorphism of $\mathbf{P}({\rm V})$.

\end{document}